\documentclass[12pt]{amsart}
\usepackage[mathscr]{eucal}
\usepackage{amssymb, amsfonts, ytableau, amscd, comment, color}
\usepackage{graphicx}
\usepackage{stmaryrd}

\theoremstyle{plain} 
\newtheorem{thm}{Theorem}[section]

\newtheorem{lem}[thm]{Lemma}
\newtheorem{prop}[thm]{Proposition}

\theoremstyle{definition}

\newtheorem{ex}[thm]{Example}

\theoremstyle{remark}
\newtheorem{rem}[thm]{Remark}

\numberwithin{equation}{section}

\def\NN{{\mathbb N}}

\def\fp{{\mathfrak p}}

\def\fS{{\mathfrak S}}
\def\fm{{\mathfrak m}}

\def\cF{\mathcal F}
\def\cL{\mathcal L}

\def\Ass{\operatorname{Ass}}

\def\Image{\operatorname{Im}}
\def\Id{\operatorname{Id}}

\def\too{\longrightarrow}

\def\rank{\operatorname{rank}}

\def\Im{\operatorname{Im}}

\def\Ker{\operatorname{Ker}}
\def\Cok{\operatorname{Coker}}

\def\SYT{\operatorname{SYT}}
\def\Tab{\operatorname{Tab}}

\def\Tor{\operatorname{Tor}}

\def\ISp{I^{\rm Sp}}

\def\too{\longrightarrow}

\def\chara{\operatorname{char}}

\def\<{{\langle}}
\def\>{{\rangle}}

\def\sgn{\operatorname{sgn}}

\title[{Minimal free resolutions of the Specht ideals of  $(n-d,d)$}]{Elementary construction\\ of the minimal free resolution \\ of the Specht ideal of shape $(n-d,d)$}

\author{Kosuke Shibata}
\address{National Institute of Technology, Yonago College, Yonago,Tottori, 683-8502, JAPAN}
\email{shibata@yonago.kosen-ac.jp}
\author{Kohji Yanagawa}
\address{Department of Mathematics, Kansai University, Suita, Osaka 564-8680, Japan}
\email{yanagawa@kansai-u.ac.jp}
\thanks{The author is partially supported by JSPS Grant-in-Aid for Scientific Research (C) 19K03456.}

\begin{document}
\maketitle

\begin{abstract}
Let $K$ be a field with $\chara(K)=0$. For a partition $\lambda$ of $n \in \NN$, let $I^{\rm Sp}_\lambda$ be the ideal of $R=K[x_1,\ldots,x_n]$ generated by all Specht polynomials of shape $\lambda$.  These ideals have been studied from several points of view (and under several names). Using advanced tools of the representation theory, Berkesch Zamaere et al. \cite{BGS} constructed a minimal free resolution  of $\ISp_{(n-d,d)}$ except differential maps.  
The present paper constructs the differential maps, and also gives an elementary proof of the result of \cite{BGS}. 
\end{abstract}

\section{Introduction}
Throughout this paper, $K$ is a field with $\chara(K)=0$.
For a positive integer $n$,  a {\it partition} of $n$ is a sequence $\lambda = (\lambda_1, \ldots, \lambda_l)$ of integers with $\lambda_1 \ge \lambda_2 \ge \cdots \ge \lambda_l \ge 1$ and $\sum_{i=1}^l \lambda_i =n$.  
The notation $\lambda \vdash n$ means that $\lambda$ is a partition of $n$. 
A partition $\lambda$ is represented by its Young diagram.  
The {\it Young tableau} of shape $\lambda$ is a bijective filling of the boxes of the Young diagram of  $\lambda$
by the integers in $\{1, \ldots, n\}$. 
Let  $\Tab (\lambda )$ be the set of tableaux of shape $\lambda$.
For a tableau $T \in \Tab (\lambda)$, its {\it Specht polynomial} $f_T \in R:=K[x_1, \ldots, x_n]$ is the product of all $x_i-x_j$ such that $i$ and $j$ are in the same column of $T$ and $j$ is in a lower position than $i$. 
For example, a partition $(4,2,1)$ is represented as $
\ytableausetup{centertableaux,boxsize=0.5em}\ydiagram {4,2,1}$,  and 
$$
T=\ytableausetup{mathmode, boxsize=1.3em}
\begin{ytableau}
4 & 2 & 1& 7   \\
3 & 5    \\
6 \\
\end{ytableau}
$$
is a tableau of shape $(4,2,1)$. In this case, its Specht polynomial  is $f_T=(x_4-x_3)(x_3-x_6)(x_4-x_6)(x_2-x_5).$

The {\it vector space} spanned by $\{ f_T \mid T \in \Tab(\lambda) \}$ has the natural $\fS_n$-module structure, where $\fS_n$ is the $n$-th symmetric group. In fact, this is isomorphic to the {\it Specht module} $V_\lambda$, which is  crucial in the representation theory of $\fS_n$.  
The {\it ideal}  
$$\ISp_\lambda =(f_T \mid T \in \Tab(\lambda)) \subset R$$ 
is called the {\it Specht ideal} of shape $\lambda$. 
These ideals have been studied from several points of view (and under several names and characterizations), so the present authors had sometimes ``re-proved" existing results in their preceding papers (\cite{SY1,Y}). The Cohen-Macaulay Specht ideals are classified as follows. (Assuming the results in \cite{LL}, the proof in \cite{Y} can be largely simplified. See also \cite{MW} for the characteristic free results.)

\begin{thm}[c.f. {\cite[Proposition~2.8 and Corollary~4.4]{Y}}]
\label{prev paper main}
$R/\ISp_\lambda$ is Cohen--Macaulay if and only if  one of the following conditions holds.  
$$
\text{(1)} \  \lambda =(n-d, 1, \ldots, 1), \qquad 
\text{(2)} \ \lambda =(n-d,d), \qquad  
\text{(3)}  \ \lambda =(d,d,1).  
$$
\end{thm}

The minimal free resolution of $\ISp_\lambda$ in the  case (1) is treated in the joint paper \cite{WY} with J. Watanabe.  In the cases (2) and (3),  C. Berkesch Zamaere et al.  \cite{BGS} determined the $\fS_n$-module structure of  $ \Tor_i^R(K, R/\ISp_\lambda)$, especially  they gave the graded Betti numbers 
$$\beta_{i,j}(R/\ISp_\lambda) :=\dim_K [\Tor_i^R(K, R/\ISp_\lambda)]_j.$$
A crucial point of their result is that 
\begin{itemize}
\item[($*$)] $\beta_{i,j}(\ISp_{(n-d,d)})\ne 0$ implies $j=d+i$ (resp. $j=d+i+1$) if $i \le n-2d$ (resp. $i > n-2d$). 
\end{itemize}

In the previous paper \cite{SY2}, the authors constructed  minimal free resolutions of  $\ISp_{(n-2,2)}$ and  $\ISp_{(d,d,1)}$ explicitly using Specht modules.  In these cases (and the case $\ISp_{(n-1,1)}$), the fact $(*)$ is easy, and this is a key point of \cite{SY2}.  Extending this result, the present paper constructs a minimal free resolution of $\ISp_{(n-d,d)}$ for general $d$.  
 
Contrary to \cite{BGS},  the present paper gives the differential maps of our resolution explicitly. 
Moreover,  while \cite{BGS} uses  highly  advanced tools of the representation theory (rational Cherednik algebras, Jack polynomials, etc), we only use  the basic theory of Specht modules.   
Our proof is self-contained, in other words, we give a new  and elementary proof  of $(*)$.  

For example, the minimal free resolution $\cF_\bullet^{(4,3)}$ of $\ISp_{(4,3)}$ is as follows. 
 $$0 \too V_{\ytableausetup{boxsize=0.25em} 
\begin{ytableau}
{}&\\
\\
\\
\\
\\
\\
\end{ytableau}}\otimes_K R(-7) \stackrel{\partial_{3}}{\too} 
V_{\ytableausetup{boxsize=0.25em} 
\begin{ytableau}
{} & \\
 {} & \\
\\
\\
\\
\end{ytableau}} \otimes_K R(-6) \stackrel{\partial_{2}}
{\too} 
V_{\ytableausetup{boxsize=0.25em} 
\begin{ytableau}
{} & {} & \\
 {} & {} & \\
\\
\end{ytableau}} \otimes_K R(-4) \qquad \qquad  \qquad \qquad \quad $$
\begin{equation}\label{F^(4,2)} \qquad \qquad \qquad \qquad \qquad \qquad \qquad 
 \stackrel{\partial_{1}}{\too} 
V_{\ytableausetup{boxsize=0.2em} 
\begin{ytableau}
{} & {} & {} & \\
 {} & {} & \\
\end{ytableau}} \otimes_K R(-3) \stackrel{\partial_{0}}{\too} \ISp_{(4,3)} 
\too 0. 
\end{equation}
 Here the degree of an element in $V_\lambda$ is 0, so we have $[V_\lambda \otimes R(-i)]_j \cong (R_{j-i})^{\oplus \dim V_\lambda}$.  
The differential maps of $\cF_\bullet^{(4,3)}$ are precisely described in Examples~\ref{(4,3)} below. 
In general, $\cF_\bullet^{(n-d,d)}$ consists of the $d$-linear strand and the $(d+1)$-linear strand (more precisely, $\cF_\bullet^{(n-1,1)}$, which is 1-linear, is the exception). 
The former ends at the $(n-2d)$-th step of the resolution, and the latter  starts from the $(n-2d+1)$-th step (this is what $(*)$ says).  
The Young diagrams in the $d$-linear strand are obtained by successively moving boxes from the top right to the bottom left, and those in the  $(d+1)$-linear strand are obtained by successively moving boxes from the second row to the bottom left.

The composition of this paper is as follows. \S2 collects basic definitions and preliminary results. 
In \S3, we give the construction of our resolution $\cF_\bullet^{(n-d,d)}$, and   \S4 mainly treats the proof of  the statement $(*)$, from which we can easily show that  $\cF_\bullet^{(n-d,d)}$ is a minimal free resolution. 
So if one can assume $(*)$, which is given by \cite{BGS}, one can skip the most part of this section. 
Our proof is an induction on $n$ using the branching rule of Specht modules.  

\section{Preliminaries}
In this section, we briefly explain Specht modules and related notions. See \cite[Chapter 2]{Sa} for details.

For a partition $\lambda=(\lambda_1,\ldots , \lambda_l)$, we sometimes use ``exponential notation''. For example,  $(4,3^2,2,1^3)$ means $(4,3,3,2,1,1,1)$.  For convenience, a partition like  $(4,3)$ sometimes denoted by $(4,3,1^0)$. 
If $\lambda = (\lambda_1, \ldots, \lambda_l)$, then $\Tab(\lambda)$ can be simply written as $\Tab(\lambda_1, \ldots, \lambda_l)$. 
Given any set $A$, let $S(A)$ be the set of all permutations of $A$.
If a tableau $T\in \Tab(\lambda)$ with $\lambda \vdash n$ has columns $C_1,\ldots ,C_{\lambda_1} \subset [n]:=\{1,2, \ldots, n\}$, then we call  
$C(T):=S(C_1)\times \cdots \times S(C_{\lambda_1})$ the {\it column-stabilizer} of $T$.
For $T, T'\in \Tab(\lambda)$, $T$ and $T'$ are {\it row equivalent}, if the corresponding rows of $T$ and $T'$ contain the same elements. For $T\in \Tab(\lambda)$, the {\it tabloid} $\mbox{\boldmath $\{$}T\mbox{\boldmath $\}$}$ of $T$ is defined by 
$$\mbox{\boldmath $\{$}T\mbox{\boldmath $\}$}:=\{T'\in \Tab(\lambda)\, |\, T \, \, {\rm and} \, \, T' \, \, {\rm are\, \,  row \, \, equivalent}\},$$
 and the {\it polytabloid} of $T$ is defined by
$$e(T):=\displaystyle \sum_{\sigma\in C(T)}\sgn(\sigma) \, \sigma \mbox{\boldmath $\{$}T\mbox{\boldmath $\}$}.$$

The vector space 
$$V_\lambda:=\displaystyle \sum_{T\in \Tab(\lambda)}Ke(T)$$
becomes an $\mathfrak{S}_n$-module in the natural way, and it is called the {\it Specht module} of $\lambda$.
The Specht modules $V_\lambda$ are irreducible, and $\{ V_\lambda \mid \lambda \vdash n\}$ forms a complete list of  the irreducible representations of $\mathfrak{S}_n$ (recall that $\chara(K)=0$ now).

In the previous section, we defined the Specht polynomial $f_T \in R=K[x_1, \ldots, x_n]$. 
Since $\mathfrak{S}_n$ acts on $R$, the vector subspace
$$\sum_{T\in \Tab(\lambda)} Kf_T$$
is also an $\mathfrak{S}_n$-module.
Moreover, the map
\begin{equation}\label{Isom}
V_{\lambda} \too \sum_{T\in \Tab(\lambda)} Kf_T, \qquad \quad e(T) \longmapsto f_T .
\end{equation}
is well-defined, and gives an isomorphism as $\mathfrak{S}_n$-modules.

The set $\{e(T)\, |\, T\in \Tab(\lambda)\}$ is linearly dependent. 
For linear relations among them, a {\it Garnir relation} is fundamental. See \cite[\S 2.6]{Sa}. 
We say a tableau $T$ is {\it standard}, if all columns (resp. rows) are increasing from top to bottom (resp. from left to right). Let $\SYT(\lambda)$ be the set of standard Young tableaux of shape $\lambda$. It is a classical result that 
$\{e(T)\, |\, T\in \SYT(\lambda)\}$ is a basis of $V_\lambda$   (c.f. \cite[Theorem 2.6.5]{Sa}). 



\begin{lem}\label{cancel} 
We have that $\beta_{i,j}(\ISp_{(n-d,d)}) \ne 0$ implies $j=d+i, d+i+1$, and  
$$\beta_{i, d+i}(\ISp_{(n-d,d)})-\beta_{i-1,d+i}(\ISp_{(n-d,d)}) = 
\begin{cases}
\dim_K  V_{(n-d-i, d , 1^{i})} & (i \le n-2d), \\
0 & (i=n-2d+1), \\
-\dim_K V_{(d-1, n-d-i+1 , 1^i)} & (i > n-2d+1), 
\end{cases}
$$
\end{lem}


\begin{proof}
In \cite{BGS} (see also \cite{SY1}), the Hilbert series of $R/\ISp_{(n-d,d)}$ is given by 
$$H(R/\ISp_{(n-d,d)}, t)= \frac{1+h_1 t +h_2 t^2+\cdots + h_d t^d }{(1-t)^d}$$
with 
$$
h_i = \begin{cases}
\binom{n-d+i-1}{i} & \text{if $1 \le i \le d-1$,} \\
\binom{n-1}{d-2} &  \text{if $ i = d$.}
\end{cases}
$$
Since $R/\ISp_{(n-d,d)}$ is Cohen--Macaulay,  the above Hilbert series states that the regularity of $R/\ISp_{(n-d,d)}$ is $d$ (cf., \cite[Corollary B.4.1]{Va}), that is, the regularity of $\ISp_{(n-d,d)}$ is $d+1$. 
Since $\ISp_{(n-d,d)}$ is generated by degree $d$ elements, we have the first assertion of the lemma.  

By \cite[Lemma~4.1.13]{BH}, we have 
\begin{eqnarray*}
\frac{\sum_{i,j}(-1)^i\beta_{i,j}(\ISp_{(n-d,d)})t^j}{(1-t)^n}
=\frac{1}{(1-t)^n}-H(R/\ISp_{(n-d,d)},t),
\end{eqnarray*}
and hence
\begin{eqnarray*}
\sum_{i,j}(-1)^i\beta_{i,j}(\ISp_{(n-d,d)})t^j
=1-(1-t)^{n-d}H(R/\ISp_{(n-d,d)},t). 
\end{eqnarray*}
Comparing the $(d+i)$--degree parts of the both sides of the above equation, we see that
\begin{eqnarray*}
&&(-1)^i (\beta_{i, d+i}(\ISp_{(n-d,d)})-\beta_{i-1,d+i}(\ISp_{(n-d,d)}) ) \\
&=& \left( \sum^{d-1}_{k=0} (-1)^{d+i-k-1}\binom{n-d}{d+i-k}h_k \right)+ (-1)^{i-1}\binom{n-d}{i}h_d\\
&=&\left( \sum^{d-1}_{k=0} (-1)^{d+i-k-1}\binom{n-d}{d+i-k}\binom{n-d+k-1}{k}\right) + (-1)^{i-1}\binom{n-d}{i}\binom{n-1}{d-2}\\
&=& (-1)^i  \binom{n-1}{d+i}\binom{d+i-1}{i}+ (-1)^{i-1}\binom{n-d}{i}\binom{n-1}{d-2}\\
&=&(-1)^{i}\frac{n!(n-2d-i+1)}{(n-d-i)!(d-1)!i!(n-d+1)(d+i)},
\end{eqnarray*}
where the third equality follows from the equation (5.5) in the proof of \cite[Proposition~5.3.14]{V}.

By the hook formula (\cite[Theorem~3.10.2]{Sa}), we have 
\begin{eqnarray*}
\dim_K  V_{(n-d-i, d , 1^{i})} =\frac{n!(n-2d-i+1)}{(n-d-i)!(d-1)!i!(n-d+1)(d+i)}
\end{eqnarray*}
for $i\leq n-2d$, and
\begin{eqnarray*}
\dim_K V_{(d-1, n-d-i+1 , 1^i)}  =\frac{n!(2d-n+i-1)}{(n-d-i)!(d-1)!i!(n-d+1)(d+i)}
\end{eqnarray*}
for $i > n-2d+1$.
Now the second assertion is clear. 
\end{proof}

\begin{rem}
The above computation of the Hilbert series is also treated in the authors' previous paper \cite{SY1}. While \cite{SY1} gives extra results (e.g. the characteristic free-ness of the Hilbert series, and the ``branching rule" for Specht ideals, which is a prototype of the argument in \S4 of the present paper), the computation in \cite{BGS} is much more direct and quicker than \cite{SY1}. 
\end{rem}

\section{Construction of the free resolution}
For  the Specht ideal  $\ISp_{(n-d,d)}$, we define the chain complex 
$$
\cF_\bullet^{(n-d,d)}:0 \too F_{n-d-1} \stackrel{\partial_{n-d-1}}{\too}  F_{n-d-2} \stackrel{\partial_{n-d-2}}{\too}  \cdots \stackrel{\partial_2}{\too}  F_1 \stackrel{\partial_1}{\too} F_0 \too 0
$$
of graded free $R$-modules as follows (if there is no danger of confusion, we simply denote $\cF_i^{(n-d,d)}$ as $F_i$). 
Here  
$$
F_i = V_{(n-d-i, d , 1^{i})} \otimes_K R(-d-i)
$$
for $0 \le i \le n-2d$, and
$$
F_i = V_{(d-1, n-d-i , 1^{i+1})} \otimes_K R(-d-i-1)
$$
for $n-2d+1\leq i \leq n-d-1$.


\begin{rem}
The second assertion of Lemma~\ref{cancel} states that 
\begin{equation}\label{cancel2}
\beta_{i, d+i}(\ISp_{(n-d,d)})-\beta_{i-1,d+i}(\ISp_{(n-d,d)}) = 
\begin{cases}
\rank \cF^{(n-d,d)}_i & (i \le n-2d), \\
0 & (i=n-2d+1), \\
-\rank \cF^{(n-d,d)}_{i-1} & (i > n-2d+1), 
\end{cases}
\end{equation}
where $\rank F$ denotes the rank of  a free module $F$. 
\end{rem}


First, we define the differential map  $\partial_i: F_i \to F_{i-1}$ for $1\leq i \leq n-2d$.  For a tableau 
$$
T:= 
\ytableausetup
{mathmode, boxsize=2.9em}
\begin{ytableau}
a_1 & b_2 &  \cdots & b_{d}  & \cdots & b_{n-d-i}  \\
a_2 & c_2 &  \cdots & c_{d}  \\
a_3\\
\vdots \\
a_{i+2}
\end{ytableau}
\in \Tab(n-d-i, d , 1^{i})
$$
and an integer $j$ with $1\leq j \leq i+2$, set
$$
T^j:= 
\ytableausetup
{mathmode, boxsize=2.9em}
\begin{ytableau}
a_1 & b_2 &  \cdots & b_{d}  & \cdots & b_{n-d-i} & a_j  \\
a_2 & c_2 &  \cdots & c_{d}  \\
a_3\\
\vdots \\
a_{j-1}\\
a_{j+1} \\
\vdots \\
a_{i+2}
\end{ytableau}
\in \Tab(n-d-i+1, d , 1^{i-1}). 
$$
Then, for $1\leq i \leq n-2d$, we set
$$
\partial_i(e(T)\otimes 1):=\sum_{j=1}^{i+2} (-1)^{j-1}e(T^j)\otimes x_{a_j}\in V_{(n-d-i+1, d , 1^{i-1})}\otimes_K R(-d-i+1)=F_{i-1}. 
$$

Recall that $e(\sigma T)=\sgn(\sigma)e(T)$ for $\sigma \in C(T)$. It is easy to check that the maps $\partial_i$ defined above  satisfies $\partial_i(e(\sigma T)\otimes 1)=\sgn(\sigma) \, \partial_i(e(T)\otimes 1)$ for  all $\sigma \in C(T)$. The same is true for $\partial_i$ for $i > n-2d$ defined below. 
However, this is not enough.  
Since $\{e(T)\, |\, T\in \Tab(\lambda)\}$ is linearly dependent, the well-definedness of $\partial_i$ is non-trivial, 
and we will prove this Proposition~\ref{wdCM} below.

To define $\partial_{n-2d+1}$, we need further preparation. For  a tableau 
\begin{equation}\label{T for d-1}
T:= 
\ytableausetup
{mathmode, boxsize=3.3em}
\begin{ytableau}
a_1 & b_2 &  \cdots & b_{d-1}  \\
a_2 & c_2 &  \cdots & c_{d-1}  \\
a_3\\
\vdots \\
a_{n-2d+4}
\end{ytableau}
\in  \Tab( (d-1)^2 , 1^{n-2d+2})
\end{equation}
and integers $j,k$ with $1\leq j<k \leq n-2d+4$, set
$$
T_{j,k}:= 
\ytableausetup
{mathmode, boxsize=2.1em}
\begin{ytableau}
\vdots & b_2 &  \cdots & b_{d-1}  & a_j  \\
\vdots & c_2 &  \cdots & c_{d-1}  & a_k\\
\vdots 
\end{ytableau}
\in \Tab(d^2, 1^{n-2d}),$$
 where the first column is the ``transpose'' of 
$$\ytableausetup
{boxsize=3.4em}\begin{ytableau}
a_1 & a_2 & \cdots & a_{j-1} & a_{j+1} & \cdots  & a_{k-1} & a_{k+1} &\cdots & a_{n-2d+4} \end{ytableau}.$$
Then we set 
\begin{eqnarray*}
\partial_{n-2d+1}(e(T)\otimes 1)&:=&\sum_{1\leq j <k \leq n-2d+4} (-1)^{j+k-1}e(T_{j,k})\otimes x_{a_j}x_{a_k} \\
&\in& V_{(d^2,1^{n-2d})}\otimes_K R(-n+d)=F_{n-2d}. 
\end{eqnarray*}

Finally, we  define $\partial_i$ for $ n-2d+2 \le i \le n-d-1$, but this case is a bit more complicated. For  a tableau 
$$
T:= 
\ytableausetup
{mathmode, boxsize=3em}
\begin{ytableau}
a_1 & b_2 &  \cdots & b_{n-d-i}    &\cdots &b_{d-1}  \\
a_2 & c_2 &  \cdots & c_{n-d-i} \\
a_3\\
\vdots \\
a_{i+3}
\end{ytableau}
$$
in $\Tab(d-1, n-d-i, 1^{i+1})$ and an integer $j$ with $1 \le j \le i+3$, set 
$$
T_j:= 
\ytableausetup
{mathmode, boxsize=3.6em}
\begin{ytableau}
a_1 & b_2 &  \cdots & b_{n-d-i}  &   b_{n-d-i+1}  &\cdots & b_{d-1}  \\
a_2 & c_2 &  \cdots & c_{n-d-i} & a_j \\
\vdots \\
a_{j-1}\\
a_{j+1} \\
\vdots \\
a_{i+3}
\end{ytableau} 
$$
in $\Tab(d-1, n-d-i+1, 1^i)$. 
Then we set
\begin{equation}\label{dif for dd}
\partial_i(e(T) \otimes 1) := \sum_{j=1}^{i+3} \sum_{\sigma \in H}  (-1)^{j-1} e(\sigma(T_j)) \otimes x_{a_j} \in V_{(d-1, d-i+2, 1^{i-1})} \otimes_K R(-d-i) =F_{i-1},
\end{equation}
where $H$ is the set of the transpositions $(b_{n-d-i+1} \ b_{n-d-i+k})$ for $k \ge 2$. 

\begin{ex}\label{(4,3)}
Recall that the ``form" of  $\cF_\bullet^{(4,3)}$ was given in Introduction.  
The differential maps are as follows.   
\begin{eqnarray*}
\partial_3 \bigl ( e \bigl( \, 
\ytableausetup
{mathmode, boxsize=1em}
\begin{ytableau}
1 & 2 \\
3   \\
4 \\
5 \\
6 \\
7
\end{ytableau}
\, \bigr)  \otimes 1 \bigr) 
&=& 
e \bigl( \, 
\ytableausetup
{mathmode, boxsize=1em}
\begin{ytableau}
3 &2  \\
4 &1\\
5 \\
6 \\
7
\end{ytableau}
  \, \bigr )
\otimes x_1 
 - 
e \bigl( \, 
\ytableausetup
{mathmode, boxsize=1em}
\begin{ytableau}
1 &2  \\
4 &3\\
5 \\
6 \\
7
\end{ytableau}
  \, \bigr )
\otimes x_3
+e \bigl( \, 
\ytableausetup
{mathmode, boxsize=1em}
\begin{ytableau}
1 &2  \\
3 &4\\
5 \\
6 \\
7
\end{ytableau}
  \, \bigr )
\otimes x_4 
 \\
 & & -  
e \bigl ( \, 
\ytableausetup
{mathmode, boxsize=1em}
\begin{ytableau} 
1 &2  \\
3 &5\\
4 \\
6 \\
7
\end{ytableau}
\, \bigr )\otimes x_5   + 
e \bigl ( \, 
\ytableausetup
{mathmode, boxsize=1em}
\begin{ytableau}
1 &2  \\
3 &6\\
4 \\
5 \\
7
\end{ytableau}
\, \bigr) \otimes x_6 -
e \bigl( \, 
\ytableausetup
{mathmode, boxsize=1em}
\begin{ytableau}
1 &2  \\
3 &7\\
4 \\
5 \\
6
\end{ytableau}
\, \bigr)
\otimes x_7
\end{eqnarray*}
(in this case, $H =\{ \varepsilon \}$), 
\begin{eqnarray*}
\partial_2 \bigl( e\bigl ( \, 
\ytableausetup
{mathmode, boxsize=1em}
\begin{ytableau}
1 & 2 \\
3  & 4 \\
5 \\
6 \\
7
\end{ytableau}
\, \bigr )  \otimes 1 \bigr ) 
&=& e \bigl( \, 
\ytableausetup
{mathmode, boxsize=1em}
\begin{ytableau}
5 & 2 & 1 \\
6  & 4 & 3 \\
7 
\end{ytableau}
\, \bigr )  \otimes x_1x_3
- e \bigl( \, 
\ytableausetup
{mathmode, boxsize=1em}
\begin{ytableau}
3 & 2 &  1\\
6  & 4 & 5\\
7 
\end{ytableau}
\, \bigr  )  \otimes x_1x_5 \\
& & \qquad  \vdots \\
& & \qquad  \vdots \\
 & & 
- e\bigl ( \, 
\ytableausetup
{mathmode, boxsize=1em}
\begin{ytableau}
1 & 2 & 5 \\
3  & 4 & 7\\
6 
\end{ytableau}
\, \bigr )  \otimes x_5x_7
+ e \bigl( \, 
\ytableausetup
{mathmode, boxsize=1em}
\begin{ytableau}
1 & 2 &  6\\
3  & 4 & 7\\
5 
\end{ytableau}
\,\bigr  )  \otimes x_6x_7,
\end{eqnarray*}
and 
\begin{eqnarray*}
\partial_1 \bigl( e\bigl ( \, 
\ytableausetup
{mathmode, boxsize=1em}
\begin{ytableau}
1 & 2 &  3\\
4  & 5 & 6\\
7
\end{ytableau}
\, \bigr )  \otimes 1 \bigr ) 
&=& e \bigl( \, 
\ytableausetup
{mathmode, boxsize=1em}
\begin{ytableau}
4 & 2 & 3& 1\\
7  & 5 & 6 
\end{ytableau}
\, \bigr )  \otimes x_1
- e \bigl( \, 
\ytableausetup
{mathmode, boxsize=1em}
\begin{ytableau}
1 & 2 & 3& 4\\
7  & 5 & 6 
\end{ytableau}
\, \bigr  )  \otimes x_4 \\
 & & 
+ e\bigl ( \, 
\ytableausetup
{mathmode, boxsize=1em}
\begin{ytableau}
1 & 2 & 3& 7\\
4  & 5 & 6 
\end{ytableau}
\, \bigr )  \otimes x_7.
\end{eqnarray*}
\end{ex}

\begin{rem}
(1) The complex $\cF_\bullet^{(n-2,2)}$ has been constructed in \cite{SY2}, but the convention is not the same. 
The complex in \cite{SY2} is a free resolution of the residue ring $R/\ISp_{(n-2,2)}$, while the present version  is a free resolution of the ideal $\ISp_{(n-2,2)}$. 
Similarly, \cite{SY2} gave a minimal free resolution $\cF^{(d,d,1)}_\bullet$ of the residue ring $R/\ISp_{(d,d,1)}$.   In the present paper,  $\cF^{(d,d,1)}_\bullet$ means the corresponding resolution of $\ISp_{(d,d,1)}$.  

(2)  Except the last step, the resolution $\cF_\bullet^{(n-2,2)}$ is 2-linear. In this sense, the differential maps $\partial_i$ of $\cF_\bullet^{(n-d,d)}$ for $1 \le i \le n-2d$ (i.e., the differential maps of the $d$-linear strand) have been studied in \cite{SY2}.  On the other hand, the resolution $\cF_\bullet^{(d,d,1)}$ is $(d+2)$-linear. However, the construction of $\partial_i$ of $\cF_\bullet^{(n-d,d)}$ for $i \ge n-2d+2$  (i.e., the differential maps of the $(d+1)$-linear strand) is same as that of $\cF_\bullet^{(d,d,1)}$. So, to the $(d+1)$-linear strand of our  $\cF_\bullet^{(n-d,d)}$, we can apply arguments for $\cF_\bullet^{(d,d,1)}$ discussed in \cite{SY2}. One of the essential problems of the present paper is $\partial_{n-2d+1}$, which connect two linear strands.    
\end{rem}

\begin{prop}\label{wdCM}
The maps $\partial_i$ are well-defined for all $i$.
\end{prop}

\begin{proof}
The well-definedness of $\partial_i$ for $i \le n-2d$ (resp. $i \ge n-2d+2$) can be shown by the same way as \cite[Theorem~4.2]{SY2}  for  $\cF_\bullet^{(n-2,2)}$  (resp. \cite[Theorem~6.2]{SY2} for  $\cF_\bullet^{(d,d,1)}$). 
So only $\partial_{n-2d+1}$ remains. To prove this, we use the following linear maps 
$$\varphi: V_{((d-1)^2, 1^{n-2d+2})} \too V_{(d, d-1, 1^{n-2d+1})} \otimes_K R_1$$
 $$\psi: V_{(d,d-1, 1^{n-2d+1})} \too V_{(d^2, 1^{n-2d})} \otimes_K R_1$$
defined by 
$$
\varphi(e(T)):=\sum_{i \ge1} (-1)^{i-1}e(T^i)\otimes x_{a_i}\in V_{(d,d-1, 1^{n-2d+1})} \otimes_K R_1
$$
for 
$$
T= 
\ytableausetup
{mathmode, boxsize=2em}
\begin{ytableau}
a_1 & b_2 &   \cdots  & b_{d-1}   \\
a_2 & c_2 & \cdots   & c_{d-1}\\
a_3\\
\vdots
\end{ytableau} \in \Tab((d-1)^2, 1^{n-2d+2}),
$$
and 
$$
\psi(e(T')):=\sum_{j \ge 1} (-1)^{j-1}e(T'_j)\otimes x_{a_j}\in V_{(d^2, 1^{n-2d})} \otimes_K R_1
$$
for 
$$
T'= 
\ytableausetup
{mathmode, boxsize=2em}
\begin{ytableau}
a_1 & b_2 &   \cdots  & b_{d-1} & b_d  \\
a_2 & c_2 & \cdots   & c_{d-1}\\
a_3\\
\vdots
\end{ytableau} \in \Tab(d,d-1, 1^{n-2d+1}). 
$$
The well-definedness of $\varphi$ (resp. $\psi$) is the same thing as that of $\partial_i$ for $i \le n-2d$ (resp. $i \ge n-2d+2$). Finally, consider the linear map $\mu: R_1 \otimes_K R_1 \to R_2$ defined by
$$
\mu(x_j \otimes x_i)=\frac12 x_ix_j.
$$  

Identifying the degree $(n-d+2)$-part of  $F_{n-2d+1}= V_{((d-1)^2, 1^{n-2d+2})} \otimes_K R(-n+d-2)$ with $V_{((d-1)^2, 1^{n-2d+2})}$,  the composition 
\begin{eqnarray*}
V_{((d-1)^2, 1^{n-2d+2})} \stackrel{\varphi}{\too} V_{(d, d-1, 1^{n-2d+1})} \otimes_K R_1 
&\stackrel{\psi \otimes \Id}{\too}& V_{(d^2, 1^{n-2d})} \otimes_K R_1\otimes_K R_1 \\ 
&\stackrel{\Id \otimes \mu}{\too}& V_{(d^2, 1^{n-2d})} \otimes_K R_2
\end{eqnarray*}
coincides with the degree $(n-d+2)$-part $[\partial_{n-2d+1}]_{n-d+2}$ of $\partial_{n-2d+1}$. Here, for example, we will compute the  $-\otimes x_{a_1} \otimes x_{a_2}$ and $-\otimes x_{a_2} \otimes x_{a_1}$
parts of  $(\psi \otimes \Id) \circ \varphi (e(T))$ for the above $T$.  
\begin{eqnarray*}
&&(\psi \otimes \Id) \circ \varphi (e(T)) \\
&=& (\psi \otimes \Id) 
 \bigl ( e \bigl( \, 
\ytableausetup
{mathmode, boxsize=1.9em}
\begin{ytableau}
a_2 & \cdots & b_{d-1} & a_1\\
a_3  & \cdots & c_{d-1} \\
\vdots 
\end{ytableau}
\, \bigr)  \otimes x_{a_1} -
e \bigl( \, 
\ytableausetup
{mathmode, boxsize=1.9em}
\begin{ytableau}
a_1 & \cdots & b_{d-1} & a_2\\
a_3  & \cdots & c_{d-1} \\
\vdots 
\end{ytableau}
\, \bigr)  \otimes x_{a_2} + \cdots \bigr) \\
&=&
e \bigl( \, 
\ytableausetup
{mathmode, boxsize=1.9em}
\begin{ytableau}
a_3 & \cdots & b_{d-1} & a_1\\
a_4  & \cdots & c_{d-1} &a_2\\
\vdots 
\end{ytableau}
\, \bigr)   \otimes x_{a_2} \otimes x_{a_1} -
e \bigl( \, 
\ytableausetup
{mathmode, boxsize=1.9em}
\begin{ytableau}
a_3 & \cdots & b_{d-1} & a_2\\
a_4  & \cdots & c_{d-1} & a_1\\
\vdots 
\end{ytableau}
\, \bigr)   \otimes x_{a_1} \otimes x_{a_2} + \cdots \\
&=&
e \bigl( \, 
\ytableausetup
{mathmode, boxsize=1.9em}
\begin{ytableau}
a_3 & \cdots &  b_{d-1} & a_1\\
a_4  & \cdots &  c_{d-1} &a_2\\
\vdots 
\end{ytableau}
\, \bigr)   \otimes x_{a_2} \otimes x_{a_1} 
+
e \bigl( \, 
\ytableausetup
{mathmode, boxsize=1.9em}
\begin{ytableau}
a_3 & \cdots & b_{d-1} & a_1\\
a_4  & \cdots & c_{d-1} & a_2\\
\vdots 
\end{ytableau}
\, \bigr)   \otimes x_{a_1} \otimes x_{a_2} + \cdots.
\end{eqnarray*}
Now  we have shown that  $[\partial_{n-2d+1}]_{n-d+2}$ is well-defined. 
Since $\partial_{n-2d+1}$ is determined by $[\partial_{n-2d+1}]_{n-d+2}$,  it is also well-defined. 
\end{proof}

\begin{prop}
$\cF_\bullet^{(n-d,d)}$ is actually a chain complex. 
\end{prop}

\begin{proof}
It suffices to show that $\partial_i \partial_{i+1}(e(T) \otimes 1  )=0$ for all $i$ and $T$. The cases $i \ne n-2d, n-2d+1$ can be shown by direct computation. More precisely, for $i < n-2d$ (resp. $i > n-2d+1$), one can check it by the same way as $\cF^{(n-2,2)}_\bullet$ 
(resp. $\cF^{(d,d,1)}_\bullet$) discussed in \cite{SY2}.  Consider the case  $i=n-2d+1$. Take $e(T) \otimes 1 \in F_{n-2d+2}$, where 
$$
T= 
\ytableausetup
{mathmode, boxsize=2em}
\begin{ytableau}
a_1 & b_2 &   \cdots  & b_{d-2} & b_{d-1}  \\
a_2 & c_2 & \cdots   & c_{d-2}\\
a_3\\
\vdots
\end{ytableau} \in \Tab(d-1, d-2, 1^{n-2d+3}). 
$$
Note that  $\partial_{n-2d+1} \partial_{n-2d+2}(e(T) \otimes 1)$ can be decomposed as 
$$\partial_{n-2d+1} \partial_{n-2d+2}(e(T) \otimes 1) =\sum_{i < j<k} \alpha_{i,j,k} \otimes x_i x_j x_k,$$
for $\alpha_{i,j,k} \in V_{(d^2, 1^{n-2d})}$, and 
\begin{eqnarray*}
\alpha_{i,j,k} &=& (-1)^{i-1}(-1)^{(j-1)+(k-1)-1} e((T_i)_{j,k})+(-1)^{j-1}(-1)^{i+(k-1)-1}e((T_j)_{i,k}) \\
                    && \qquad \qquad \qquad \qquad  +(-1)^{k-1}(-1)^{i+j-1}e((T_k)_{i,j})\\
&=& (-1)^{i+j+k}( e((T_i)_{j,k})-e((T_j)_{i,k}) + e((T_k)_{i,j})).
\end{eqnarray*}
The right most boxes of $(T_i)_{j,k}, (T_j)_{i,k}$ and $(T_k)_{i,j}$ are the following 
$$
(T_i)_{j,k}=
\ytableausetup
{mathmode, boxsize=2.1em}
\begin{ytableau}
\none[\cdots]  & b_{d-1} & a_j\\
\none[\cdots]  & a_i & a_k\\
\end{ytableau}, 
\qquad 
 (T_j)_{i,k}=
\ytableausetup
{mathmode, boxsize=2.1em}
\begin{ytableau}
\none[\cdots]  & b_{d-1} & a_i\\
\none[\cdots]  & a_j & a_k\\
\end{ytableau} ,
\qquad 
(T_k)_{i,j}=
\ytableausetup
{mathmode, boxsize=2.1em}
\begin{ytableau}
\none[\cdots]  & b_{d-1} & a_i\\
\none[\cdots]  & a_k & a_j\\
\end{ytableau} 
$$ 
(these tableau are same in the remaining boxes). 
By a Garnir relation, we have $ e((T_i)_{j,k})-e((T_j)_{i,k}) + e((T_k)_{i,j})=0$, and hence $\partial_{n-2d+1} \partial_{n-2d+2}(e(T) \otimes 1) =0$. 

That  $\partial_{n-2d} \partial_{n-2d+1}(e(T) \otimes 1)=0$ can be shown in a similar way (when $d=2$, it has been  shown in \cite{SY2}). 
\end{proof}

For the complex $\cF^{(n-d,d)}_\bullet$, the augmentation map $\partial_0 : F_0 \to \ISp_{(n-d,d)}$ is defined by 
$$\partial_0 : F_0 = V_{(n-d,d)} \otimes_K R(-d) \ni e(T)\otimes 1 \longmapsto  f_T \in \ISp_{(n-d,d)}.$$ 
Clearly, $\partial_0$ is surjective and  well-defined  (recall the map \eqref{Isom}).    

\begin{lem}\label{aug}
With the above notation, we have $\partial_0 \partial_1=0$. 
\end{lem}

\begin{proof}
The case $d=2$ is shown in \cite{SY2}, and the proof there works for the general case.  (Since  $\cF_\bullet^{(n-2,2)}$ in \cite{SY2} is a resolution of $R/\ISp_{(n-2,2)}$, our  $\partial_0 \partial_1=0$ corresponds to $\partial_1  \partial_2 =0$ there.)
\end{proof}

We define the $\mathfrak{S}_n$-action on $\cF_\bullet^{(n-d,d)}$. For 
$F_i=V_\lambda \otimes_K R(-j)$ (here $\lambda$ is a suitable partition of  $n$ and $j$ is a suitable integer) and  $\sigma \in\fS_n$, set $\sigma(v\otimes f):=\sigma v \otimes \sigma f$. 
It is easy to check that $\partial_i$ is  an  $\mathfrak{S}_n$-homomorphism for all $i$. 

\begin{lem}\label{injective}
With the above notation, the degree $j$-part $[\partial_i]_j : [F_i]_j \to [F_{i-1}]_j$ is injective for all $i \ge 1$.   
Similarly, the degree $d$-part of the augmentation map $\partial_0 : F_0 \to \ISp_{(n-d,d)}$ is injective (actually, bijective). 
\end{lem}

The corresponding fact has been shown in \cite{SY2}, but we will explain it again.  

\begin{proof}
  Since $[F_i]_j \cong V_\lambda$  is irreducible as an $\fS_n$-module and  $[\partial_i]_j$ is  an  $\fS_n$-homomorphism, it is either the zero map or an injection. However, it is clearly non-zero.   
\end{proof}

By Lemmas~\ref{aug} and \ref{injective}, we see that the former half $F_{n-2d}\to \cdots \to F_1 \to F_0 \to 0$ of $\cF_\bullet^{(n-d,d)}$ is a subcomplex of  a minimal free resolution $P_\bullet$ of $\ISp_{(n-d,d)}$ such that each $F_i$ is a direct summand of $P_i$ for each $0 \le i \le n-2d$.

\section{The proof of the main theorem}
The following is the main result of the present paper, and this section is devoted to its proof.    

\begin{thm}\label{main}
The complex $\cF_\bullet^{(n-d,d)}$ is a minimal free resolution  of $\ISp_{(n-d,d)}$.  
\end{thm}

The theorem is easy when $n=2d$. 

\begin{prop}\label{(d,d)}
If $n=2d$, Theorem~\ref{main} holds, that is, $\cF^{(d,d)}_\bullet$ is a minimal free resolution of $\ISp_{(d,d)}$. 
\end{prop}

\begin{proof}
By the equation \eqref{cancel2}, we have $\beta_{1, d+1} (\ISp_{(d,d)})= \beta_{0, d+1} (\ISp_{(d,d)})=0$, and hence $\beta_{i,d+i} (\ISp_{(d,d)})=0$ for all $i \ge1$. So, by the first statement of Lemma~\ref{cancel}, $\beta_{i,j} (\ISp_{(d,d)}) \ne 0$ for  $i \ge1$ implies  $j = d+i+1$. By Lemmas~\ref{aug} and \ref{injective}, $\cF_\bullet^{(d,d)}$ is a subcomplex of a minimal free resolution $P_\bullet$ of $\ISp_{(d,d)}$ such that $\cF^{(d,d)}_i$ is a direct summand of $P_i$ for all $i$, but \eqref{cancel2} implies that $P_i = \cF^{(d,d)}_i$ for all $i \ge 0$, and hence $P_\bullet = \cF^{(d,d)}_\bullet$.   
\end{proof}

In the general case, the proof of Theorem~\ref{main} is much more complicated, and preparation is required. 
For $n,d \in \NN$ with $n \ge 2d$, we construct  the  chain complex $\cL_\bullet^{n,d}$ of graded free $R$-modules as follows. 
$$0 \too \cL_{n-2d}^{n,d} \too \cdots \too  \cL_1^{n,d} \too \cL_0^{n,d} \too 0,$$
$$\cL_i^{n,d} = V_{(n-d-i,d,1^i )} \otimes_K R(-d-i), $$
and  the differential maps are given by the same way as the former half of $\cF^{(n-d,d)}_\bullet$.  Clearly, $\cL_\bullet^{n,d}$ is the $d$-linear strand of $\cF^{(n-d,d)}_\bullet$, and it is actually a chain complex.   
If $n=2d$, then $\cL_i^{n,d}=0$ for $i \ne 0$, and Theorem~\ref{main} has been proved in this case (Proposition~\ref{(d,d)}). 
So we mainly treat the case $n>2d$. Since the case $n=2d+1$ is still somewhat special, we first assume that $n>2d+1$.  

\begin{prop}\label{reduction}
For $n,d \in \NN$ with $n > 2d+1$, the following are equivalent. 
\begin{itemize}
\item[(1)] Theorem~\ref{main} holds in this case. 
\item[(2)] $H_i(\cL^{n,d}_\bullet)=0$ for $0< i <n-2d$. 
\end{itemize}
\end{prop}

\begin{proof}
(1) $\Rightarrow$ (2): Clear. 

(2) $\Rightarrow$ (1): 
First, we prove that  $H_1(\cL^{n,d}_\bullet)=0$ (or even 
$[H_1(\cL^{n,d}_\bullet)]_{d+2}=0$) implies  $H_0(\cL^{n,d}_\bullet) \cong \ISp_{(n-d,d)}$. 
Note that $\cL^{n,d}_\bullet$ is a subcomplex of a minimal free resolution $P_\bullet$ of  $\ISp_{(n-d,d)}$ with $\cL^{n,d}_0  = P_0$. Hence if   $\cL^{n,d}_1  = P_1$, then $H_0(\cL^{n,d}_\bullet)  = H_0(P_\bullet)\cong \ISp_{(n-d,d)}$. 
By \eqref{cancel2}, we see that 
\begin{eqnarray*}
\dim_K [\cL^{n,d}_1]_{d+1}=\rank \cF_1^{(n-d,d)} &=&\beta_{1,d+1}(\ISp_{(n-d,d)}) - \beta_{0,d+1}(\ISp_{(n-d,d)})\\
&=&\beta_{1,d+1}(\ISp_{(n-d,d)}) =\dim _K [P_1]_{d+1},
\end{eqnarray*}  
and hence $ [\cL^{n,d}_1]_{d+1}=[P_1]_{d+1}$. 
By the first statement of Lemma~\ref{cancel}, if $\cL^{n,d}_1 \ne P_1$ then $P_1 = \cL^{n,d}_1 \oplus R(-d-2)^m $ for some $m >0$. By \eqref{cancel2}, we have $P_2 = \cL^{n,d}_2 \oplus R(-d-2)^m \oplus R(-d-3)^l$ 
($l$ can be 0). Consider the exact sequence 
$$0 \too \cL_\bullet^{n,d} \too P_\bullet \too  P_\bullet/\cL_\bullet^{n,d} \too 0.$$
Both $P_2/\cL_2^{n,d}$ and $P_1/\cL_1^{n,d}$ have the direct summand $R(-d-2)^m$. Since $P_\bullet$ is minimal, we have $\Image \partial_2^{P_\bullet} \subset \fm P_1$, and hence $\partial_2^{P_\bullet}$ induces the zero map between these two $R(-d-2)^m$. 
Hence $[H_2(P_\bullet/\cL_\bullet^{n,d})]_{d+2}\ne 0$, but the above exact sequence induces $[H_1(\cL_\bullet^{n,d})]_{d+2}=[H_2(P_\bullet/\cL_\bullet^{n,d})]_{d+2} \ne 0$. This is a contradiction, and we have $m=0$.     
Hence $H_0(\cL^{n,d}_\bullet) \cong \ISp_{(n-d,d)}$. 

If (2) holds, $\cL^{n,d}_\bullet$ coincides with  the first $n-2d$ steps of $P_\bullet$ (since  $H_{n-2d-1}(\cL^{n,d}_\bullet)=0$, the kernel of $\partial: \cL_{n-2d-1}^{n,d} \too \cL^{n,d}_{n-2d-2}$ 
coincides with $\partial(\cL_{n-2d}^{n,d} )$). In particular,  $\beta_{n-2d,n-d+1}( \ISp_{(n-d,d)})=0$. 
By \eqref{cancel2}, we have $\beta_{n-2d+1, n-d+1} (\ISp_{(n-d, d)})
=0$, and hence $\beta_{i,d+i} (\ISp_{(n-d,d)})=0$ for all $i > n-2d$.   
Summing up,  $\beta_{i,j} (\ISp_{(n-d,d)}) \ne 0$ and $i \le n-2d$ (resp. $i > n-2d$) implies $j=d+i$ (resp. $j=d+i+1$). 
 By an argument similar to the last step of the proof Proposition~\ref{(d,d)}, we have $P_\bullet = \cF^{(n-d,d)}_\bullet$.   
\end{proof}

The next lemma can be proved by a similar way to Proposition~\ref{reduction}, and we leave the details to the reader.

\begin{prop}\label{reduction2}
For $n,d \in \NN$ with $n = 2d+1$, the following are equivalent. 
\begin{itemize}
\item[(1)] Theorem~\ref{main} holds in this case. 
\item[(2)] $[H_1(\cL^{n,d}_\bullet)]_{d+2}=0$. 
\end{itemize}
\end{prop}

We regard $\fS_{n-1}$ as the symmetric group on $\{1, \ldots, n-1\}$. Regarding  the $\fS_n$-module $V_{(n-d-i, d,1^i)}$ as an $\fS_{n-1}$-module in the natural way, we get  the {\it restriction} $V_{(n-d-i, d,1^i)} \! \downarrow^{\fS_n}_{\fS_{n-1}}$.  By the {\it branching rule} (\cite[Theorem~2.8.3]{Sa}), we have 
$$V_{(n-d-i, d,1^i)} \! \downarrow^{\fS_n}_{\fS_{n-1}} \cong 
V_{(n-d-i-1, d, 1^i)}  \oplus V_{(n-d-i, d-1,1^i)} \oplus V_{(n-d-i, d, 1^{i-1})}.$$
Here, if $(n-d-i-1, d, 1^i)$ is not a partition (i.e., if  $n-d-i-1 < d$), we set $V_{(n-d-i-1, d, 1^i)} =0$. Similarly, if $i=0$ then $V_{(n-d-i, d, 1^{i-1})}=0$. 

\bigskip

Even in \cite{SY2}, the complex  $\cF^{(n-1,1)}_\bullet$ is skipped, but this case is very simple. Note that $\ISp_{(n-1,1)} =(x_1-x_i \mid 2 \le i \le n )$ is a linear complete intersection, and $\beta_i(\ISp_{(n-1,1)})=\beta_{i,i+1}(\ISp_{(n-1,1)})=\binom{n-1}{i+1}$ for all $0 \le i \le n-2$. On the other hand,  $\cF^{(n-1,1)}_\bullet$ is 1-linear and $\rank \cF^{(n-1,1)}=\binom{n-1}{i+1}$ for all $0 \le i \le n-2$. So,  by an argument similar to the last step of the proof Proposition~\ref{(d,d)},  we can show that $\cF^{(n-1,1)}_\bullet$ is a minimal free resolution of $\ISp_{(n-1,1)}$. 
Moreover, \cite[Theorem~3.2]{SY2} states that $\cF^{(n-2,2)}_\bullet$ is a minimal free resolution of $\ISp_{(n-2,2)}$. So we may assume that $d\ge 3$. 

Recall that $\{ \, e(T) \mid T \in \SYT(n-d-i, d,1^i) \, \}$ forms a basis of $V_{(n-d-i, d,1^i)}$. For $T \in \SYT(n-d-i, d,1^i)$, the box \fbox{$n$} is located in one of the following positions
\begin{itemize}
\item[(i)]  the right end of the first row  (if $n-d-i >d$), 
\item[(ii)] the right end of the second row,  
\item[(iii)] the bottom of the first column (if $i \ge 1$). 
\end{itemize}

We consider the following subspaces of $V_i:=V_{(n-d-i, d,1^i)}$. 
$$U_i := \< \, e(T) \mid T \in \SYT(n-d-i, d,1^i), \, \text{\fbox{$n$} is in the position (i) } \>$$
$$W_i := \< \, e(T) \mid T \in \SYT(n-d-i, d,1^i), \, \text{\fbox{$n$} is in the positions (i) or (ii)}\, \>$$
Note that $V_i/W_i$ has  a basis 
\begin{equation}\label{basis of V/W}
\{\, \overline{e}(T) \mid T \in \SYT(n-d-i, d,1^i), \, \text{\fbox{$n$} is in the position (iii) } \},
\end{equation}
where $\overline{e}(T)$ denotes the natural image of $e(T)$.   
These vector spaces are closed under the $\fS_{n-1}$-action, and we have 
$$U_i \cong V_{(n-d-i-1, d,1^i)}, \quad W_i/U_i \cong V_{(n-d-i, d-1,1^i)} \quad  \text{and} \quad V_i/W_i \cong 
V_{(n-d-i, d,1^{i-1})}$$ 
as $\fS_{n-1}$-modules. 
%

In the above notation, we have $\cL^{n,d}_i =V_i \otimes_K R(-d-i)$.   If $z \in U_i \otimes_K R(-d-i) \subset \cL^{n,d}_i$, 
then it is easy to see that $\partial^{\cL^{n,d}_\bullet}_i (z) \in  U_{i-1} \otimes_K R(-d-i+1) \subset \cL^{n,d}_{i-1}$. Similarly, $z \in W_i \otimes_K R(-d-i)$ implies $\partial^{\cL^{n,d}_\bullet}_i (z) \in  W_{i-1} \otimes_K R(-d-i+1)$. 
Hence, if we set $A_i := U_i \otimes_K R(-d-i)$ and $B_i := W_i \otimes_K R(-d-i)$ for $i \ge 0$,  $A_\bullet$ and $B_\bullet$ are subcomplexes of $\cL^{n,d}_\bullet$.    
Clearly, these are complexes of $\fS_{n-1}$-modules.

Set $R':=K[x_1, \ldots, x_{n-1}]$. 
Then $\cL^{n-1,d}_\bullet$ and $\cL^{n-1,d-1}_\bullet$ are chain complexes of graded free $R'$-modules. 

\begin{lem}\label{ABC}
With the above notation, we have the following. 
\begin{itemize}
\item[(1)] $A_\bullet \cong \cL^{n-1,d}_\bullet \otimes_{R'} R$ 
\item[(2)] $B_\bullet/A_\bullet$ is isomorphic to a truncation of $\cL^{n-1,d-1}_\bullet(-1) \otimes_{R'} R$, where  $(-1)$ denotes the degree shift for graded modules.   More precisely, if we remove  the ``last term" $\cL^{n-1, d-1}_{n-2d+1}(-1) \otimes_{R'} R$ from  $\cL^{n-1,d-1}_\bullet(-1) \otimes_{R'} R$, it is isomorphic to $B_\bullet/A_\bullet$.  
\item[(3)] $\cL^{n,d}_\bullet/B_\bullet \cong \cL^{n-1,d}_{\bullet-1}(-1) \otimes_{R'} R$.  
\end{itemize}
\end{lem}

\begin{proof}
(1) We have $A_i \cong V_{(n-d-i-1, d,1^i)} \otimes_K R(-d-i) \cong \cL_i^{n-1, d} \otimes_{R'} R$ 
for $0 \le i \le n-2d-1$, and $A_i=\cL_i^{n-1, d} \otimes_{R'} R=0$ for $i \ge n-2d$.  
Under these isomorphisms, the differential maps clearly coincide.  

(2) We have 
$$\cL_i^{n-1, d-1} \otimes_{R'} R\cong V_{(n-d-i,d-1,1^i)}\otimes_K R(-d-i+1)$$
for $0 \le i \le (n-1)-2(d-1)=n-2d+1$. Similarly, 
$$B_i/A_i \cong V_{(n-d-i, d-1, 1^i)}\otimes_K R(-d-i)$$
for $0 \le i  \le n-2d$ (note that $\cL_i^{n,d}=0$ for $i > n-2d$). The differential maps clearly coincide, and (2) holds.  

(3) We have $\cL^{n,d}_0/B_0=0$ and 
$$\cL^{n,d}_i/B_i \cong (V_i/W_i)\otimes_K R(-d-i) \cong V_{(n-d-i, d,1^{i-1})} \otimes_K  R(-d-i),$$
for $1 \le i \le n-2d$. Similarly, 
$$ \cL^{n-1,d}_{i-1} \otimes_{R'} R \cong V_{(n-d-i, d,1^{i-1})} \otimes_K  R(-d-i+1)$$
for $1 \le i \le n-2d$.  Let $\overline{e}(T)$ be an element of the basis \eqref{basis of V/W} of $V_i/W_i$. 
If the first column of $T$ has the entries $a_1,a_2, \ldots, a_{i+2} \, (=n)$ from top to bottom, we have 
$$\partial_i^{\cL^{n,d}_\bullet}(e(T) \otimes 1)=\sum_{j=1}^{i+2}(-1)^{j-1}(e(T^j) \otimes x_{a_j}).$$
Since $e(T_{i+2}) \in W_{i-1}$ and $\overline{e}(T_{i+2}) =0$,  we have 
$$\partial_i^{\cL^{n,d}_\bullet/B_\bullet}(\overline{e}(T) \otimes 1)=\sum_{j=1}^{i+1}(-1)^{j-1}(\overline{e}(T^j) \otimes x_{a_j}).$$
This clearly coincides with $\partial_{i-1}^{\cL^{n-1,d}_\bullet \otimes_{R'} R}$, and hence (3) holds. 
\end{proof}  

\noindent{\it The proof of Theorem~\ref{main}.} 
  We use induction on $n$.  Since the theorem has been shown for $d \le 2$ in \cite{SY2}, the theorem holds for small $n$'s. 
We may also assume that $n > 2d$ by Proposition~\ref{(d,d)}. 

Clearly, $A_i=0$ for $i  \ge  n-2d$. By the induction hypothesis and Lemma~\ref{ABC}, we have  
\begin{itemize}
\item $\cF_\bullet^{(n-d-1,d)}$ (resp. $\cF_\bullet^{(n-d,d-1)}$) is a minimal free resolution of the Specht ideal $\ISp_{(n-d-1,d)}$ (resp. $\ISp_{(n-d,d-1)}$) of $R'$. 
\item $A_\bullet$ is the first $n-2d-1$ steps of  $\cF_\bullet^{(n-d-1,d)} \otimes_{R'} R$. In particular,  $H_i(A_\bullet)=0$ for $0<i  < n-2d-1$.   
\item  $B_\bullet/A_\bullet$ is the first $n-2d$ steps of  $\cF_\bullet^{(n-d,d-1)}(-1) \otimes_{R'} R$. In particular, 
$H_i(B_\bullet/A_\bullet)=0$ for $0<i <n-2d$.  (We have $H_i(\cL_\bullet^{n-1, d-1})=0$ for $0<i \le n-2d$, but we have to exclude the case $i=n-2d$ by the truncation). 
\end{itemize}

First, we will study $H_i(B_\bullet)$ using the exact sequence 
\begin{equation}\label{AB}
0 \too A_\bullet \too B_\bullet \too B_\bullet/A_\bullet \too 0. 
\end{equation}

\medskip

\noindent{\it Claim~1.} If $n>2d+1$, then $H_i(B_\bullet) =0$ for $0 <i < n-2d$.

\medskip

\noindent{\it The proof of Claim~1.}  
By \eqref{AB} and the above observation, we have $H_i(B_\bullet)=0$ for $0 <i <n-2d-1$, and it remains to prove  $H_{n-2d-1}(B_\bullet)=0$. Since $A_{n-2d}=0$, we have  the exact sequence 
\begin{equation}\label{middle length}
0 \too  H_{n-2d}(B_\bullet) \stackrel{f}{\too} 
H_{n-2d}(B_\bullet /A_\bullet) \stackrel{g}{\too}H_{n-2d-1}(A_\bullet) \stackrel{h}{\too} H_{n-2d-1}(B_\bullet)\too 0. 
\end{equation}

By the 
above observation, 
$$H_{n-2d-1}(A_\bullet)=\Ker(\partial_{n-2d-1}^{A_\bullet}) = \Ker(\partial_{n-2d-1}^{\cF^{(n-d-1,d)}_\bullet})\otimes_{R'} R= \Image (\partial^{\cF^{(n-d-1,d)}_\bullet}_{n-2d}) \otimes_{R'} R.$$
Since  $\cF^{(n-d-1,d)}_{n-2d} \otimes_{R'} R \cong V_{((d-1)^2, 1^{n-2d+1})} \otimes_K R(-n+d-1)$, 
$H_{n-2d-1}(A_\bullet)$ is generated by its degree $(n-d+1)$-part  $$[H_{n-2d-1}(A_\bullet)]_{n-d+1} \cong V_{((d-1)^2, 1^{n-2d+1})}.$$

On the other hand, since the ``tail" of $B_\bullet/A_\bullet$ is of the form 
$$0 \too \cL_{n-2d}^{n-1,d-1}(-1)  \otimes_{R'} R \too \cL_{n-2d-1}^{n-1,d-1}(-1)  \otimes_{R'} R$$ 
by Lemma~\ref{ABC} (2) (while $\cL_{n-2d+1}^{n-1,d-1} \ne 0$  and $H_{n-2d}(\cL_\bullet^{n-1,d-1})=0$), we have 
$$H_{n-2d}(B_\bullet / A_\bullet) \cong \Ker (\partial^{\cL^{n-1, d-1}_\bullet}_{n-2d})(-1) \otimes_{R'} R \cong \Image (\partial^{\cL^{n-1, d-1}_\bullet}_{n-2d+1})(-1) \otimes_{R'} R.$$ 
Since $\cL^{n-1, d-1}_{n-2d+1}  \otimes_{R'} R \cong V_{((d-1)^2, 1^{n-2d+1})} \otimes_K R(-n+d)$, we have 
 $$[H_{n-2d}(B_\bullet / A_\bullet)]_{n-d+1} \cong V_{((d-1)^2, 1^{n-2d+1})}.$$

Since, in the sequence \eqref{middle length}, 
\[
  \begin{CD}
     [g]_{n-d+1}: \, @. [H_{n-2d}(B_\bullet / A_\bullet)]_{n-d+1}  @>>>   [H_{n-2d-1}(A_\bullet)]_{n-d+1} \\
 @.   \cong @|   \cong @| \\
 @.  V_{((d-1)^2, 1^{n-2d+1})}   @. V_{((d-1)^2, 1^{n-2d+1})} 
  \end{CD}
\]
is an $\fS_{n-1}$-homomorphism, and $H_{n-2d-1}(A_\bullet)$ is generated by $[H_{n-2d-1}(A_\bullet)]_{n-d+1}$, we have 
\begin{eqnarray*}
[g]_{n-d+1} \ne 0 \Longrightarrow \text{$[g]_{n-d+1} $ is bijective} &\Longrightarrow& 
\text{$g$ is surjective} \\
 &\Longrightarrow& \text{$h$ is injective} \Longrightarrow  H_{n-2d-1}(B_\bullet) =0.
\end{eqnarray*}
 So it suffices to show that $ [g]_{n-d+1} \ne 0$. 

Since $A_{n-2d}=B_{n-2d+1}=0$, we have 
$$H_{n-2d}(B_\bullet)=\Ker \partial^{B_\bullet}_{n-2d} \qquad \text{and} \qquad H_{n-2d}(B_\bullet/A_\bullet) 
=(\partial^{B_\bullet}_{n-2d})^{-1}(A_{n-2d-1})$$ 
(both are submodules of $B_{n-2d}$), and $f$ is the inclusion map. 

As defined in the construction of the differential map of $\cF_\bullet^{(n-d,d)}$, 
for a tableau $$T=\ytableausetup
{mathmode, boxsize=2.2em}
\begin{ytableau}
a_1 & b_2 & \cdots & b_{d-1} \\
a_2 & c_2 & \cdots & c_{d-1} \\
a_3 \\
\vdots 
\end{ytableau}  \in \Tab((d-1)^2, 1^{n-2d+1})$$
we have  the tableau $T^i  \in \Tab(d, d-1, 1^{n-2d})$ for each $1 \le i \le n-2d+3$.   
Adding \fbox{$n$} to the right end of the second row of $T^i$, we get $\widetilde{T}^i \in \Tab(d^2,1^{n-2d})$. 
Set 
$$
\alpha  := \sum_{i \ge 1} (-1)^{i-1}e(\widetilde{T^i}) \otimes x_{a_i} \in [W_{n-2d} \otimes_{R'} R(-n+d)]_{n-d+1} =[B_{n-2d}]_{n-d+1}.  
$$

We want to compute $\partial^{B_\bullet}_{n-2d}(\alpha)$.   If $T',T'',T''' \in \Tab(d+1, d,1^{n-2d-1})$ are tableaux whose  right most 3 entries are  
$$
T'=
\ytableausetup
{mathmode, boxsize=2em}
\begin{ytableau}
\none[\cdots] &  a_i & a_j\\
\none[\cdots]  &  n   
\end{ytableau},
\qquad 
T''=
\ytableausetup
{mathmode, boxsize=2em}
\begin{ytableau}
\none[\cdots]   & a_j & a_i \\
\none[\cdots]   & n   
\end{ytableau}
\qquad 
\text{and}
\qquad
T'''=
\ytableausetup
{mathmode, boxsize=2em}
\begin{ytableau}
\none[\cdots]   & a_i & n \\
\none[\cdots]   & a_j  
\end{ytableau},
$$
and the rests are same, then $e(T')=e(T'')+e(T''')$ and $e(T''') \in A_{n-2d-1}$. Hence it is easy to see that
$\partial^{B_\bullet}_{n-2d}(\alpha) \in  A_{n-2d-1}$ and $\alpha \in H_{n-2d}(B_\bullet/A_\bullet)$. 
Moreover, the term  
$$e\Biggl( \ \ytableausetup
{mathmode, boxsize=2.2em}
\begin{ytableau}
a_3 & b_2 & \cdots & b_{d-1} & a_1 & n\\
a_4 & c_2 & \cdots & c_{d-1}  & a_2 \\
\vdots 
\end{ytableau} \ \Biggr)
\otimes x_{a_1}x_{a_2}
$$
survives in $\partial^{B_\bullet}_{n-2d}(\alpha)$. Hence $\partial^{B_\bullet}_{n-2d}(\alpha) \ne 0$, and it  means that $\alpha \not \in H_{n-2d}(B_\bullet)=\Im f= \Ker g$ and $g(\alpha) \ne 0$. 
We are done. \qed

\bigskip

\noindent{\it Claim~2.} If $n=2d+1$, then $[H_1(B_\bullet)]_{d+2} = 0$. 

\medskip

\noindent{\it The proof of Claim~2.}  By the same argument as in Claim~1, we have an injection $f : H_1(B_\bullet) \to H_1(B_\bullet/A_\bullet)$, and we see that $[H_1(B_\bullet/A_\bullet)]_{d+2} \cong V_{((d-1)^2,1^2)}$. 
Since $ V_{((d-1)^2,1^2)} $ is a simple $\fS_{n-1}$-module, and $f$ is an $\fS_{n-1}$-homomorphism, 
the degree $(d+2)$-part $[f]_{d+2}$ is either  a bijection or the zero map (then  $[H_1(B_\bullet)]_{d+2}=0$).  So it suffices to show that $[H_1(B_\bullet)]_{d+2} \ne [H_1(B_\bullet/A_\bullet)]_{d+2}$. 

As in Claim~1, we have $H_1(B_\bullet)=\Ker \partial_1^{B_\bullet}$ and $H_1(B_\bullet/A_\bullet)=(\partial_1^{B_\bullet})^{-1}(A_0)$. So it suffices to show that there is some 
$\alpha \in [B_1]_{d+2}$ with $0 \ne \partial^{B_\bullet}_1(\alpha) \in A_0$. However, we can find such $\alpha$ in the same way as in Claim~1. 
\qed

\medskip

Let us return to the main line of the proof. The proof is divided into two cases. In both cases, we use the exact sequence 
\begin{equation}\label{B-L}
0 \too B_\bullet \too \cL^{n,d}_\bullet \too \cL^{n,d}_\bullet/B_\bullet \too 0. 
\end{equation}

\noindent{\it Case 1.} First, we consider the case $n > 2d+1$. It suffices to prove that  $H_i(\cL^{n,d}_\bullet)=0$ for $0<i <n-2d$ by Proposition~\ref{reduction}.  
 By Lemma~\ref{ABC} (3) and the induction hypothesis, we have $H_i(\cL^{n,d}_\bullet/B_\bullet)=0$ for $1<i <n-2d$. 
By \eqref{B-L}, Claim~1 and the above observation, we have $H_i(\cL^{n,d}_\bullet)=0$ for $1<i <n-2d$, and it remains to show that  $H_1(\cL^{n,d}_\bullet)=0$.  By Lemma~\ref{ABC} (3) and the induction hypothesis, we have $H_1(\cL^{n,d}_\bullet/B_\bullet) \cong H_0(\cL^{n-1, d}_\bullet)(-1) \otimes_{R'} R \cong (\ISp_{(n-d-1,d)})(-1) \otimes_{R'} R \cong (\ISp_{(n-d-1,d)}R)(-1)$, where $\ISp_{(n-d-1,d)}R$ is the extension of the ideal  $ \ISp_{(n-d-1,d)} \subset R'$ to $R$.  Hence we have the following exact sequence 
$$0 \too H_1( \cL^{n,d}_\bullet ) \stackrel{\varphi}{\too}  (\ISp_{(n-d-1,d)}R)(-1) \stackrel{\psi}{\too} H_0(B_\bullet) \too H_0(\cL^{n,d}_\bullet) \too 0.$$
By Lemma~\ref{ABC} (1), (2) and the induction hypothesis, $H_0(A_\bullet) \cong \ISp_{(n-d-1, d)}R$ and $H_0(B_\bullet/A_\bullet)  \cong (\ISp_{(n-d, d-1)}R)(-1)$. 
Since $H_1(B_\bullet/A_\bullet)=0$ now, \eqref{AB} yields 
$$0 \too \ISp_{(n-d-1, d)}R \too H_0(B_\bullet) \too (\ISp_{(n-d, d-1)}R)(-1) \too 0,$$
and hence $\Ass_R (H_0(B_\bullet)) =\{ (0) \}$. Note that $\fS_n$ acts on $H_1( \cL^{n,d}_\bullet )$, but does not act on $\ISp_{(n-d-1,d)}R$.  So $\varphi$ cannot be an isomorphism.  
Since $\ISp_{(n-d-1,d)}R \subset R$,  $H_1( \cL^{n,d}_\bullet ) \ne 0$ implies that $\Image \psi=\Cok \varphi$ has an associated prime $\fp \ne (0)$. Since $\Image \psi$ is a submodule of $H_0(B_\bullet)$, $\fp$ is also an associated prime of $H_0(B_\bullet)$, and this is a contradiction. 
Hence we have $H_1( \cL^{n,d}_\bullet ) = 0$, and we are done. 
\medskip

\noindent{\it Case 2.} Finally, we consider the case $n = 2d+1$. It suffices to show that $[H_1(\cL^{n,d}_\bullet)]_{d+2}=0$ by Proposition~\ref{reduction2}. By  Claim~2, the exact sequence \eqref{B-L} yields  
$\varphi : H_1( \cL^{n,d}_\bullet ) \to H_1(\cL^{n,d}_\bullet/ B_\bullet)$ such that 
$$[\varphi]_{d+2} : [H_1( \cL^{n,d}_\bullet )]_{d+2} \too [H_1(\cL^{n,d}_\bullet/ B_\bullet)]_{d+2}$$
is injective.  Recall that $H_1(\cL^{n,d}_\bullet/ B_\bullet) \cong H_0(\cL_\bullet^{n-1,d} ) (-1) \otimes_{R'} R$. 
Now $\cF_\bullet^{(n-d-1,d)} \, (=\cF_\bullet^{(d,d)})$ is a minimal free resolution of $\ISp_{(d,d)} \subset R'$, $\cF_0^{(d,d)} \cong \cL_0^{n-1,d}$ and $\cL_i^{n-1,d}=0$ for $i >0$, hence we have 
$$H_1(\cL^{n,d}_\bullet/ B_\bullet) \cong H_0(\cL_\bullet^{n-1,d})\otimes_{R'} R =\cL_0^{n-1,d}\otimes_{R'} R=\cF_0^{(d,d)}\otimes_{R'} R.$$  
Since there is a graded submodule $N$ of $\cF_0^{(d,d)}$ such that $\cF_0^{(d,d)}/N \cong \ISp_{(d,d)}$ and $N$ is generated by $N_{d+2}$, there is a submodule $L \, (\cong N(-1) \otimes_{R'} R)$ of $H_1(\cL^{n,d}_\bullet/B_\bullet)$ such that $L$ is generated by $L_{d+3}$ and 
$$H_1(\cL^{n,d}_\bullet/B_\bullet)/L \cong \ISp_{(d, d)}(-1) \otimes_{R'} R \cong (\ISp_{(d, d)}R)(-1).$$ 
If  $[H_1(\cL^{n,d}_\bullet)]_{d+2} \ne 0$, then we have a  subideal $J \subset \ISp_{(d, d)}R$ with 
$$J(-1) \cong \Image \varphi/ (L \cap \Image \varphi)$$ 
and  $J_{d+1} \ne 0$. 
Since $\fS_n$ acts on $J_{d+1} \, (\cong [\Image \varphi]_{d+2} \cong [H_1(\cL^{n,d}_\bullet)]_{d+2})$, the ideal $J' := (J_{d+1})$ is a symmetric ideal.   
For  a subsets $F \subset [n-1]$ (resp. $F \subset [n]$), set 
$$P_F' := (x_i-x_j \mid i,j \in F) \subset R' \quad \text{(resp.}  \, P_F := (x _i-x_j \mid i,j \in F) \subset R \,)$$
to be the prime ideal. 
Since 
$$\ISp_{(d,d)}=\bigcap_{\substack{F \subset [n-1] \\ \# F =d+1}}P'_F,$$
we have 
$$\ISp_{(d,d)}R=\bigcap_{\substack{F \subset [n-1] \\ \# F =d+1}}P_F.$$
Since $J' \subset \ISp_{(d,d)}R$ is a symmetric ideal, we have 
$$J' \subset \bigcap_{\substack{F \subset [n] \\ \# F =d+1}}P_F=\ISp_{(d,d,1)}.$$
However, we have $J_{d+1}=J'_{d+1} \subset [\ISp_{(d,d,1)}]_{d+1}=0$. This is a contradiction. 

\qed

\section*{Acknowledgements} 
We are grateful to Professors Stephen Griffeth,  Steven V Sam and Alexander Woo for giving us valuable information on the preceding works on Specht ideals, especially \cite{BGS}. We also thank Professors Mitsuyasu Hashimoto and  Satoshi Murai for stimulating discussion. Finally, we thank the anonymous reviewer for his/her careful reading of our manuscript and valuable comments. 


\begin{thebibliography}{8}
\bibitem{BGS}
 C. Berkesch Zamaere, S. Griffeth, and S.V. Sam, Jack polynomials as fractional quantum
Hall states and the Betti numbers of the $(k + 1)$-equals ideal, Comm. Math. Phys. {\bf  330} (2014), 415--434.

\bibitem{BH}
 W. Bruns and J. Herzog, Cohen-Macaulay rings, rev. ed., Cambridge Studies in Advanced Mathematics {\bf 39}, 1998.

\bibitem{E}D. Eisenbud, Commutative Algebra with a view toward algebraic geometry, Graduate texts in Mathematics 150, Springer-Verlag, 1995.  




\bibitem{LL}
S.-Y.R.~Li and W.C.W.~Li, 
Independence numbers of graphs and generators of ideals, \textit{Combinatorica} \textbf{1} (1981) 55--61.

\bibitem{MW} C. McDaniel and J. Watanabe, 
Principal radical systems, Lefschetz properties and perfection of Specht Ideals of two-rowed partitions, \textit{Nagoya Math. J.} \textbf{247} (2022) 690--730.   



\bibitem{Sa} B.E. Sagan, The Symmetric Group, Representations, Combinatorial Algorithms, and Symmetric Functions, second edition, Graduate texts in Mathematics 203, Springer-Verlag,  2001.

\bibitem{SY1} K. Shibata and K. Yanagawa, Regularity of Cohen-Macaulay Specht ideals, \textit{J. Algebra} \textbf{582} (2021), 73--87. 

\bibitem{SY2}
K.~Shibata and K.~Yanagawa,
Elementary construction of minimal free resolutions of the Specht ideals of shapes $(n-2,2)$ and $(d,d,1)$, to appear in \textit{J. Algebra its Appl.}, 
arXiv:2010.06522.




\bibitem{Va} W. V. Vasconcelos, Computational Methods in Commutative Algebra and Algebraic Geometry, Springer--Verlag, 1998.


\bibitem{V} R. H. Villarreal, Monomial algebras, Second edition, Monographs and research notes in mathematics, CRC Press, Boca Raton FL, 2015.

\bibitem{WY} J. Watanabe and K. Yanagawa, Vandermonde determinantal ideals, Math. Scand. {\bf 125} (2019),  179--184. 

\bibitem{Y} K. Yanagawa, When is a Specht ideal Cohen-Macaulay?   J. Commut. Algebra {\bf 13} (2021), 589--608. 
\end{thebibliography}
\end{document}